\documentclass[12pt, reqno]{amsart}
\usepackage{amssymb,amsthm,amsfonts,amsmath, amscd}

\usepackage{hyperref}
\usepackage{mathrsfs}

\newcommand\Y{\mathbb Y}
\newcommand\Z{\mathbb Z}
\newcommand\C{\mathbb C}
\newcommand\R{\mathbb R}

\newcommand\G{\mathbb G}
\newcommand\GT{{\mathbb{GT}}}

\newcommand\De{\Delta}
\newcommand\de{\delta}

\newcommand\La{\Lambda}

\newcommand\si{\sigma}

\newcommand\epsi{\varepsilon}

\newcommand\VG{\varGamma}

\renewcommand\L{\mathfrak L}
\newcommand\X{\mathfrak X}
\renewcommand\Y{\mathfrak Y}

\newcommand\Sym{\operatorname{Sym}}
\newcommand\card{\operatorname{card}}

\newcommand\sgn{\operatorname{sgn}}

\newcommand\wt{\widetilde}

\newcommand\GG{\wt\G}
\newcommand\Lan{\La^{N+1}_N}

\newcommand\p{\partial}

\renewcommand\P{\mathscr P}
\newcommand\V{\mathcal V}
\newcommand\W{\mathcal W}

\newcommand\ms{\medskip}
\renewcommand\ss{\smallskip}

\newtheorem{theorem}{Theorem}[section]
\newtheorem{proposition}[theorem] {Proposition}

\newtheorem{corollary}[theorem]{Corollary}
\newtheorem{lemma}[theorem]{Lemma}

\theoremstyle{definition}
\newtheorem{definition}[theorem]{Definition}
\newtheorem{remark}[theorem]{Remark}

\numberwithin{equation}{section}

\thanks{The present research was carried out at the Institute for Information Transmission Problems of the Russian Academy of Sciences at the expense of the Russian Science Foundation (project 14-50-00150).
}

\begin{document}

\title[]{Extended Gelfand--Tsetlin graph, its  $q$-boundary, and $q$-B-splines}

\author{Grigori Olshanski}

\date{}

\begin{abstract}

The boundary of the Gelfand--Tsetlin graph is an infinite dimensional locally compact space whose points parameterize the extreme characters of the infinite-dimensional group $U(\infty)$. The problem of harmonic analysis on the group $U(\infty)$ leads to a continuous family of probability measures on the boundary --- the so-called zw-measures. In recent work by Vadim Gorin and the author we began studying a $q$-analogue of the zw-measures. It turned out that for their construction it is necessary to introduce a novel combinatorial object --- the extended Gelfand--Tsetlin graph. In the present paper, it is proved that the Markov kernels connected with the extended Gelfand--Tsetlin graph and its $q$-boundary possess the Feller property. This property is needed for constructing a Markov dynamics on the $q$-boundary. A connection with the B-splines and their $q$-analogues is also discussed. 

\end{abstract}


\maketitle


\section{Introduction}\label{sect1}

The \emph{Gelfand--Tsetlin graph} $\GT$ is a graded graph with infinite levels $\GT_N$, $N=1,2,\dots$\,. The vertices of the $N$th level $\GT_N$ parameterize the irreducible characters of the compact unitary group $U(N)$, and the edges of $\GT$ encode the branching rule of characters under restriction from $U(N)$ to $U(N-1)$. Following Vershik and Kerov, one can define the \emph{boundary} $\p\GT$ of the graph: this is an infinite-dimensional topological space, which can be regarded as the simplest version of dual object of the group $U(\infty):=\varinjlim U(N)$. The problem of harmonic analysis for the group $U(\infty)$ leads to a family of remarkable probability measures on $\p\GT$ called \emph{zw-measures}. Each zw-measure gives rise to a determinantal point process of log-gas type. Further, one can construct continuous time Markov processes on $\p\GT$ for which the zw-measures serve as stationary distributions. See Olshanski \cite{Ols-JFA}, Borodin--Olshanski \cite{BO-AnnMath}, \cite{BO-JFA}, \cite{BO-AdvMath}, and survey papers Borodin--Olshanski \cite{BO-ECM}, Olshanski \cite{Ols-ICM}, \cite{Ols-Piter}.

The present paper is a continuation of the recent paper \cite{GO} by Vadim Gorin and the author. Our goal is to build a $q$-version of the whole theory related to the Gelfand--Tsetlin graph. In \cite{GO} it was shown that there exists a $q$-analogue of the zw-measures, but to construct them it is necessary to replace the graph $\GT$ by a novel combinatorial object, which we called the \emph{extended Gelfand--Tsetlin graph}. This graph, denoted by $\G$, contains $\GT$. The graph $\G$ violates a customary finiteness condition, which makes it impossible to define the ordinary boundary; however, a $q$-version of the boundary is well defined. 

The principal purpose of the present paper is to prepare a foundation for constructing Markov processes related to the $q$-zw-measures, but I hope that some results are of independent interest. Here is a brief description of the contents of the paper. 

Section \ref{sect2} introduces a few necessary definitions. We fix three parameters $\zeta_-<0$, $\zeta_+>0$, and $q\in(0,1)$, and introduce the \emph{two-sided $q$-lattice}
\begin{equation}\label{eq1.A}
\L:=\{\zeta_-q^n: n\in\Z\}\cup\{\zeta_+q^n: n\in\Z\}\subset\R.
\end{equation}
The vertices of $\G_N$ are identified with the $N$-point configurations (=subsets) on $\L$, and the edges of $\G$ are equipped with formal multiplicities of the form $\zeta_\pm q^n$. As explained in \cite{GO}, from these data one can define an infinite sequence $\{\Lan:N=1,2,\dots\}$ of stochastic matrices, where the $N$th matrix $\Lan$ has format $\G_{N+1}\times\G_N$, that is, $\G_{N+1}$ parameterizes the rows, and $\G_N$ parameterises the columns. These matrices are actually all that we need from the graph $\G$; in particular, the $q$-boundary $\p\G$ is entirely determined by the collection $\{\Lan:N=1,2,\dots\}$. Proposition \ref{prop2.B}, which is borrowed from \cite{GO}, is the foundation of our computations.

In Section \ref{sect3} we study the telescopic products $\La^N_K=\La^N_{N-1}\dots\La^{K+1}_K$, where $N>K\ge1$. These are also stochastic matrices; $\La^N_K$ has format $\G_N\times\G_K$.  Theorem \ref{thm3.A} provides an explicit formula for $\La^N_K$; this is a $q$-analogue of  \cite[Theorem 7.2]{BO-AdvMath} and an extension of  Theorem 1.2 from Petrov \cite{P}. Further,Theorem \ref{thm3.B} generalizes Theorem \ref{thm3.A}. 

In Section \ref{sect4} we slightly extend the definition of the matrices $\Lan$. For future application to Markov dynamics we need to deal with Feller processes on locally compact spaces, while the $q$-boundary of $\G$ is not locally compact. This obstacle is easily overcome by a completion of the spaces $\G_N$: we replace each $\G_N$ by a larger topological space $\GG_N$, which contains $\G_N$ as a dense subset. Proposition \ref{prop4.A} shows that our stochastic matrices can be extended by continuity. 

In the short Section \ref{sect5} we formulate general facts concerning the entrance boundary for an infinite chain of discrete spaces linked by stochastic matrices. This material can be presented in various ways; for our purposes it is convenient to adopt the approach of Olshanski \cite{Ols-JFA}. 

In Section \ref{sect6} we describe the boundary of the graph $\GG$ (its levels are the spaces $\GG_N$): Theorem \ref{thm6.A} says that the boundary can be identified with the set $\GG_\infty$ of two-sided bounded configurations $X$ on $\L$, infinite or finite; the boundedness condition means that if $X$ is infinite, then its points accumulate near $0$ but do not approach $\pm\infty$. The set $\GG_\infty$ has a natural structure of locally compact ultrametric space. Theorem \ref{thm6.A} is a slight generalization of Theorem 3.12 of \cite{GO}, which describes  the boundary of $\G$, but the proof is different: the approach of \cite{GO} relies on some qualitative estimates of the large-$N$ asymptotics of the matrices $\La^N_K$, while in the present paper we use exact formulas obtained in Section \ref{sect3}. In this aspect our method is closer to those of Borodin--Olshanski \cite{BO-AdvMath} and Petrov \cite{P}. 

The subject of Section \ref{sect7} is a connection between the graph $\G$ and the \emph{$q$-B-splines}, recently introduced by Simeonov and Goldman \cite{SG}. The $q$-B-splines are certain discrete analogues of the classical B-splines. We show that the $q$-B-splines with the knots on the lattice $\L$ are given by the rows of the stochastic matrices $\La^N_1$. The paper \cite{BDGO} contains (among other things)  a $q$-analogue of the classical Hermite--Genocchi formula, which relates splines to divided differences. We give a very short proof of this $q$-analogue and derive from it a technical result (Corollary \ref{cor7.B}), which is then used in Section \ref{sect9}. 

Sections \ref{sect8} and \ref{sect9} are devoted to the proof of the main results: Theorem \ref{thm9.A} and \ref{thm9.B}. They assert that the Markov kernels $\La^N_K$ and $\La^\infty_K$ related to the graph $\GG$ possess the Feller property, i.e. the corresponding contraction operators act in Banach spaces of continuous functions vanishing at infinity. The Feller property is necessary for constructing Markov processes by the method of intertwiners \cite{BO-JFA}, \cite{Ols-Piter}. 

Section \ref{sect10} contains a few remarks concerning the connection between the subject of the present paper and the B-splines.

\section{Preliminaries}\label{sect2}

\subsection{Generalities about Markov kernels}\label{sect2.1}

For more details, see e.g. Meyer \cite{Meyer}.

$\bullet$ Given two Borel spaces $\mathfrak X$ and $\mathfrak Y$, let $X$ range over $\X$ and $A$ range over Borel subsets of $\Y$ (here and below the term ``Borel space'' means a space with a distinguished sigma algebra of subsets). A \emph{Markov kernel} $\La: \X\dasharrow\Y$ is a nonnegative function $\La(X,A)$ such that the quantity $\La(X,A)$  is a Borel measurable function in the first variable and a probability measure on $\Y$ as a set function in the second variable.  We will also use the alternative notation $\La(X,dY)$. 
\ss

$\bullet$ Alternatively, $\La$ can be viewed as a Borel map $\X\to\P(\Y)$, where $\P(\,\cdot\,)$ denotes the set of probability Borel measures on a given Borel space. This map extends to an affine map $\La:\P(\X)\to\P(\Y)$.  
\ss

$\bullet$ By duality, $\La$ also determines a contractive linear operator $\mathscr B(\Y)\to \mathscr B(\X)$, where the symbol $\mathscr B(\,\cdot\,)$ denotes the Banach space of bounded Borel functions with the supremum norm on a given Borel space. More precisely, this operator is defined by 
$$
(\La F)(X)=\int_{Y\in\Y}\La(X,dY)F(Y), \qquad \textrm{where} \quad X\in\X, \quad Y\in\Y, \quad F\in\mathscr B(\Y).
$$
\ss

$\bullet$ Given two Markov kernels, $\La': \X\dasharrow\Y$ and $\La'': \Y\dasharrow \mathfrak Z$, their composition $\La'\La''$ is a Markov kernel $\X\dasharrow \mathfrak Z$ defined by
$$
(\La'\La'')(X,dZ)=\int_{Y\in\mathfrak Y}\La(X,dY)\La''(Y,dZ).
$$
Equivalently, $\La'\La''$  can be defined as the composed map $\P(\X)\to\P(\Y)\to\P(\mathfrak Z)$.

$\bullet$ If $\Y$ is a discrete space, then a Markov kernel $\La:\X\dasharrow\Y$ can be viewed as a function $\La(X,Y)$ on $\X\times\Y$, where we set $\La(X,Y):=\La(X,\{Y\})$. In this case the integral defining the corresponding operator reduces to a sum:
$$
(\La F)(X)=\sum_{Y\in\Y}\La(X,Y)F(Y).
$$

\ss

$\bullet$ If both $\X$ and $\Y$ are discrete spaces, then $\La$ is simply a stochastic matrix of format $\X\times\Y$. In the case of discrete spaces, composition of Markov kernels reduces to matrix multiplication.

\subsection{The two-sided $q$-lattice $\L$ and interlacing configurations}\label{sect2.2}
Fix parameters $q\in(0,1)$, $\zeta_+>0$ and $\zeta_-<0$, and set
$$
\L^\pm:=\{\zeta_\pm q^n: n\in\Z\}, \qquad \L:=\L^+\cup\L^-.
$$
We call $\L$ the \emph{two-sided $q$-lattice} in $\R$.

Define the \emph{intervals} in $\L$ as follows: if $a<a'$ are two points in $\L$, then
\begin{equation}\label{eq2.D}
I(a,a'):=\begin{cases} [a,a')\cap\L, & a<a'<0,\\
[a,a']\cup\L, & a<0<a',\\
(a,a']\cup\L, & 0<a<a'.
\end{cases}
\end{equation}

Note that the definition of intervals is symmetric with respect to the
reflection about $0$ combined with the switching $(\zeta_+,\zeta_-)\to (-\zeta_-,-\zeta_+)$.

Next, we introduce the \emph{extended $q$-Gelfand-Tsetlin graph}, denoted by
$\G$. Its $M$th level $\G_M$ is formed by $M$-point configurations
$A=(a_1<\dots<a_M)$ on $\L$, $M=1,2,\dots$\,. We may regard $\G_M$ as a subset of the cone (a closed Weyl chamber)
$$
\De_M:=\{(a_1,\dots,a_M)\in\R^M: a_1\le\dots\le a_M\}.
$$

\begin{definition}\label{def2.A}
We say that two configurations $X\in\G_{N+1}$ and $Y\in\G_{N}$ \emph{interlace}
if
$$
y_i\in I(x_i, x_{i+1}), \quad i=1,\dots,N.
$$
Then we write $Y\prec X$ or $X\succ Y$.
\end{definition}

By definition, each pair $X\succ Y$ forms an \emph{edge} of $\G$. 
Note that in the case $a<0<b$, the interval $I(a,b)$ contains infinitely many
points. It follows that whenever $X$ contains points of opposite signs, there
are infinitely many edges $X\succ Y$. This is a new effect: in all examples of
branching graphs studied so far, the number of edges issued from a vertex of level $N+1$ and
directed downwards to the level $N$ was always finite.

\subsection{Stochastic matrices $\La^N_K$}
For $X=(x_1<\dots<x_{N+1})\in\G_{N+1}$ and $Y=(y_1<\dots<y_N)\in\G_N$  we set
\begin{equation}\label{eq2.A}
\La^{N+1}_N(X,Y):=\begin{cases} (q;q)_N|Y|\dfrac{|V(Y)|}{|V(X)|}, & Y\prec X,\\
0, & \textrm{otherwise}. \end{cases}
\end{equation}
Here 
$$
(a;q)_N:=\prod_{i=0}^{N-1}(1-aq^i), \qquad N=0,1,2,\dots,
$$
is the standard notation for the $q$-Pochhammer symbol (see Gasper and Rahman \cite{GR}) and 
we also use the following notation: if $A=(a_1<\dots<a_M)\in\G_M$, then 
$$
|A|:=|a_1|\dots |a_M|
$$
and
$$
V(a_1,\dots,a_M):=\prod_{1\le i<j\le M}(a_i-a_j),
$$
so that 
$$
|V(A)|=(-1)^{M(M-1)/2}V(a_1,\dots,a_M)=\prod_{1\le i<j\le M}(a_j-a_i)>0.
$$ 

\begin{proposition}[see \cite{GO}]\label{prop2.A}
We have
$$
\sum_{Y\in\G_N}\La^{N+1}_N(X,Y)=1, \qquad \forall X\in\G_{N+1}, \quad N=1,2,\dots,
$$
so that $\La^{N+1}_N$ is a stochastic matrix of format $G_{N+1}\times\G_N$.
\end{proposition}

For $N>K$ we set
\begin{equation}\label{eq2.C}
\La^N_K=\La^N_{N-1}\dots\La^{K+1}_K.
\end{equation}
This is a stochastic matrix of format $\G_N\times\G_K$.

Given two real numbers $a<b$, let $\G_N[a,b]\subset\G_N$ denote the subset of configurations contained in the closed interval $[a,b]$. The set $\G[a,b]$  is finite if $0<a<b$ or $a<b<0$; otherwise it is infinite.  We say that a subset of $\G_N$ is \emph{bounded} if it is contained in $\G_N[a,b]$ with appropriate $a<b$.

\begin{proposition}\label{prop2.C}
For any pair $N>K$ and any $X\in\G_N$, the support of the probability measure $\La^N_K(X,\,\cdot\,)$ is bounded. More precisely, it is contained in $\G_K[a,b]$ provided that $X\subset[a,b]$. 
\end{proposition}

\begin{proof}
For $K=N-1$ this immediately follows from \eqref{eq2.A} and Definition \ref{def2.A}. Next, for $N-K\ge2$ we use \eqref{eq2.C}.
\end{proof}

This simple proposition plays an important role in what follows. It implies, in particular, that the operator $\La^N_K$ can be applied not only to bounded functions $F$ on $\G_K$ but, more generally, to any function which is bounded on bounded subsets.

\subsection{Schur polynomials}
Let $\Sym(N)$ denote the algebra of symmetric polynomials in $N$ variables. It has a distinguished basis formed by the Schur polynomials. We denote these polynomials by $S_{\nu\mid N}$: here the index $\nu$ is an arbitrary partition with length $\ell(\nu)$ less or equal to $N$.  

Next, we set
\begin{equation}\label{eq2.B}
\wt S_{\nu\mid N}:=\frac{S_{\nu\mid N}}{S_{\nu\mid N}(1,q,\dots q^{N-1})}.
\end{equation}

\begin{proposition}[see \cite{GO}]\label{prop2.B}
Let $N>K$  and $\nu$ be a partition with $\ell(\nu)\le K$. For every $X\in\G_N$ the following relation holds
\begin{equation}\label{eq2.B1}
\sum_{Y\in\G_K}\La^N_K(X,Y)\wt S_{\nu\mid K}(Y)=\wt S_{\nu\mid N}(X).
\end{equation}
\end{proposition}

\begin{remark}\label{rem2.A}
These relations are a kind of formula for the moments of the measure $\La^N_K(X,\,\cdot\,)$, and they characterize the kernel  $\La^N_K$ uniquely. For the graph $\GT$ there are similar relations (see \cite[(5.6)]{BO-AdvMath}), but they involve factorial Schur polynomials, not the ordinary ones. This is one of paradoxical examples when in the case of $q$-analogues the situation is simplified. 
\end{remark}

\section{Computations with the kernels $\La^N_K$}\label{sect3}

We start with the case $K=1$.

\begin{proposition}\label{prop3.A}
Let $z\in\C\setminus\R$. We have
\begin{equation}\label{eq3.A}
\sum_{y\in\bar\L}\La^N_1(X,y)\frac1{(yz^{-1};q)_N}=\prod_{x\in X}\frac1{1-x z^{-1}}, \qquad X\in\G_N.
\end{equation}
\end{proposition} 

\begin{proof}
For $K=1$, formula \eqref{eq2.B1} takes the form
\begin{equation}\label{eq3.B}
\sum_{y\in\L}\La^N_1(X,y) y^n=\frac{h_n(X)}{h_n(1,q,\dots,q^{N-1})}, \qquad X\in\G_N, \quad n=0,1,2,\dots\,.
\end{equation}
Observe that
$$
h_n(1,q,\dots,q^{N-1})=\frac{(q^N;q)_n}{(q;q)_n}
$$
(see Macdonald \cite[Chapter 1, Section 2, Ex. 3]{M}) and rewrite \eqref{eq3.B} as
\begin{equation*}
\frac{(q^N;q)_n}{(q;q)_n}\sum_{y\in\L}\La^N_1(X,y) y^n=h_n(X), \qquad n=0,1,2,\dots\,.
\end{equation*}
Assume first that $|z|$ is large, multiply the both sides by $z^{-n}$ and then sum over $n=0,1,2,\dots$\,. In the right-hand side we get the desired expression. 

Next, in the left-hand side we may interchange the order of summation. Then the interior sum will be
$$
\sum_{n=0}^\infty y^nz^{-n}\frac{(q^N;q)_n}{(q;q)_n}=\frac{(yz^{-1}q^N;q)_\infty}{(yz^{-1};q)_\infty}=\frac1{(yz^{-1};q)_N},
$$
where the first equality follows from the $q$-binomial theorem (Gasper-Rahman \cite[(1.3.2)]{GR}). This gives us the desired formula.

Finally, we extend the result to arbitrary $z\in\C\setminus\R$ by analytic continuation. 
\end{proof}

Now we extend the result of Proposition \ref{prop3.A} to the case of general $K<N$.

Let $Z=(z_1,\dots,z_K)\in(\C\setminus\R)^K$  be a $K$-tuple of pairwise distinct numbers.   We define a function $f_{Z,N,K}$ on $\G_K$ by 
\begin{equation}\label{eq3.C}
f_{Z,N,K}(Y):=\dfrac{\det\left[\dfrac1{(y_iz_j^{-1};q)_{N-K+1}}\right]_{i,j=1}^K}{V(Y)V(Z^{-1})},
\end{equation}
where 
$$
Z^{-1}:=(z_1^{-1},\dots,z_K^{-1}).
$$
Note that the right-hand side of \eqref{eq3.C} does not depend on the numeration of the points in $Y$ and $Z$. 

The assumption that the parameters $z_1,\dots,z_K$ are not real is introduced in order to avoid vanishing of the denominators in \eqref{eq3.C}. As for the assumption that the parameters are pairwise distinct, it is actually redundant and adopted for simplicity only. 

\begin{proposition}\label{prop3.B} 
Let $Z$ be as above. For $X\in\G_N$ we have
\begin{equation}\label{eq3.D}
\La^N_K f_{Z,N,K}(X)=\prod _{i=1}^K\frac{(q;q)_{N-i}}{(q;q)_{N-K}(q;q)_{K-i}}\cdot \prod_{x\in X}\prod_{j=1}^K\frac1{1-x z_j^{-1}}.
\end{equation}
\end{proposition}

In the case $K=1$ this formula reduces to that of Proposition \ref{prop3.A}. 

\begin{proof}
We first perform some formal transformations and then justify them.

By virtue of Proposition \ref{prop2.B}, we have for $X\in\G_K$
$$
\sum_{Y\in\G_K}\La^N_K(X,Y) \frac{S_{\nu\mid N}(1,q,\dots,q^{N-1})}{S_{\nu\mid K}(1,q,\dots,q^{K-1})} S_{\nu\mid K}(Y)=S_{\nu\mid N}(X).
$$
Let us multiply both sides of this equality by 
$S_{\nu\mid K}(Z^{-1})$
and then take the sum over all partitions $\nu$ with $\ell(\nu)\le K$. Then we get
\begin{multline}\label{eq3.E}
\sum_{\nu:\, \ell(\nu)\le K}\sum_{Y\in\G_K}\La^N_K(X,Y) \frac{S_{\nu\mid N}(1,q,\dots,q^{N-1})}{S_{\nu\mid K}(1,q,\dots,q^{K-1})} S_{\nu\mid K}(Y)S_{\nu\mid K}(Z^{-1})\\
=\sum_{\nu:\, \ell(\nu)\le K}S_{\nu\mid N}(X)S_{\nu\mid K}(Z^{-1}).
\end{multline}

We are going to show that \eqref{eq3.E} is equivalent to \eqref{eq3.D}. The right-hand side of \eqref{eq3.E} equals
$$
\prod_{x\in X}\prod_{j=1}^K\frac1{1-x z_j^{-1}},
$$
which agrees with the right-hand side of \eqref{eq3.D}, up to a numerical factor (the product over $i$ in \eqref{eq3.D}). 

Let us examine now the left-hand side of \eqref{eq3.E}. We interchange the order of summation, which gives 
\begin{equation}\label{eq3.F}
\sum_{Y\in\G_K}\La^N_K(X,Y) \sum_{\nu:\, \ell(\nu)\le K}\frac{S_{\nu\mid N}(1,q,\dots,q^{N-1})}{S_{\nu\mid K}(1,q,\dots,q^{K-1})} S_{\nu\mid K}(Y)S_{\nu\mid K}(Z^{-1}). 
\end{equation}
The key observation is that the ratio entering this formula is a multiplicative expression in the coordinates
$$
n_i:=\nu_i+K-i, \qquad i=1,\dots,K.
$$
Indeed, using a well-known formula for the evaluation of a Schur polynomial at a geometric progression (see Macdonald \cite[Ch. I. Section 3, Ex. 1]{M}) we get
$$
\frac{S_{\nu\mid N}(1,q,\dots,q^{N-1})}{S_{\nu\mid K}(1,q,\dots,q^{K-1})}=\prod_{i=1}^K\frac{(q;q)_{K-i}(q;q)_{N-K}}{(q;q)_{N-i}}\cdot \prod_{i=1}^K\frac{(q^{N-K+1};q)_{n_i}}{(q;q)_{n_i}}.
$$

This enables us to write the interior sum in \eqref{eq3.F} in the form
\begin{equation}\label{eq3.G}
\prod_{i=1}^K\frac{(q;q)_{K-i}(q;q)_{N-K}}{(q;q)_{N-i}}\cdot\frac1{V(Y)V(Z^{-1})}\sum_{n_1>\dots>n_K\ge0}\det[f_{n_r}(y_i)]\det[g_{n_r}(z_j)],
\end{equation}
where
$$
f_n(y):=y^n, \qquad g_n(z):=z^{-n}\frac{(q^{N-K+1};q)_n}{(q;q)_n}.
$$

Next we apply a well-known identity (all determinants are of order $K$)
$$
\sum_{n_1>\dots>n_K\ge0}\det[f_{n_r}(y_i)]\det[g_{n_r}(z_j)]=\det[h(i,j)],
$$
where
$$
h(i,j)=\sum_{n=0}^\infty f_n(y_i)g_n(z_j).
$$

In our concrete situation the last sum can be computed explicitly:
$$
h(i,j)=\sum_{n=0}^\infty y_i^nz_j^{-n} \frac{(q^{N-K+1};q)_n}{(q;q)_n}=\frac{(y_iz_j^{-1}q^{N-K+1};q)_\infty}{(y_iz_j;q)_\infty}=\frac1{(y_iz_j;q)_{N-K+1}},
$$
where the second equality follows from the $q$-binomial theorem (Gasper-Rahman \cite[(1.3.2)]{GR}). This implies that \eqref{eq3.G} can be rewritten as
$$
\prod_{i=1}^K\frac{(q;q)_{K-i}(q;q)_{N-K}}{(q;q)_{N-i}}f_{Z\mid N,K}(Y)
$$
and hence \eqref{eq3.F} (which is the left-hand side of \eqref{eq3.E}) takes the form
\begin{equation}\label{eq3.H}
\prod_{i=1}^K\frac{(q;q)_{K-i}(q;q)_{N-K}}{(q;q)_{N-i}}\sum_{Y\in\G_K}\La^N_K(X,Y)f_{Z\mid N,K}(Y).
\end{equation}

Thus, we see that the desired equality \eqref{eq3.D} is obtained  by multiplying the both sides of \eqref{eq3.E} by the numeric factor  
$$
\prod_{i=1}^K\frac{(q;q)_{N-i}}{(q;q)_{K-i}(q;q)_{N-K}}.
$$
Then it will appear in the right-hand side and will be cancelled in the left-hand side by the pre-factor from \eqref{eq3.H}. 

 To justify the above transformations we observe that all these series absolutely converge for small $z_1^{-1},\dots,z_K^{-1}$, and in the very end we may apply analytic continuation. The situation is the same as in the case when we want to compute the Stieltjes transform of a compactly supported measure with known moments: the Stieltjes kernel is a generating series for the moments, which has a finite radius of convergence, but then we may apply analytic continuation. 
\end{proof}

We are going to extract from formula \eqref{eq3.D} an explicit expression for $\La^N_K(X,Y)$ written in terms of a contour integral representation. To make the arguments clearer we examine first the simplest case $K=1$. 

For every $y\in\L$ we fix a vertical line in the complex plane $\C$, separating the points $y$ and $yq$, and oriented from top to bottom; let us denote it by $C(y)$. In other words, $C(y)=(a+\sqrt{-1}\infty, a-\sqrt{-1}\infty)$, where $a\in\R$ is chosen arbitrarily inside the interval between $y$ and $yq$.  

\begin{proposition}\label{prop3.C}
Let $X\in\G_N$ and $y\in\L${\rm;} then
\begin{equation}\label{eq3.K}
\La^N_1(X,y)=\frac{(1-q^{N-1})|y|}{2\pi\sqrt{-1}}\int_{z\in C(y)}(yz^{-1}q;q)_{N-2}\prod_{x\in X}\frac1{1-xz^{-1}}\frac{dz}{z^2}.
\end{equation}
\end{proposition}

The proof relies on the following lemma.

\begin{lemma}\label{lemma3.A}
Let $u$ and $y$ be two points from $\L$ and $N=2,3\dots$\,. Then 
$$
\frac{(1-q^{N-1})|y|}{2\pi\sqrt{-1}}\int_{z\in C(y)}\frac{(yz^{-1}q;q)_{N-2}}{(uz^{-1};q)_N}\frac{dz}{z^2}=\begin{cases} 1, & u=y,\\
0, & u\ne y. \end{cases}
$$
\end{lemma}

\begin{proof}[Proof of the lemma]
For large $|z|$ the integrand is $O(|z|^{-2})$. It follows that the integral is absolutely convergent. Moreover, we may replace $C(y)$ by any of the two closed contours $C^+(y,R)$, $C^-(y,R)$ which are defined as follows. We take a large number $R>0$; start at the point $a+\sqrt{-1}R$ (where $a\in\R$ is as above); go along the vertical line $\Re z=a$ till $a-\sqrt{-1}R$, then return to $a+\sqrt{-1}R$ along one of the semicircles $\{z: |z-a|=R, \,\pm(\Re z-a)\ge0\}$. 
The two closed contours will produce the same result because the residue of the integrand at infinity equals 0.

Let us verify the claim of the lemma for $y>0$ (for $y<0$ the argument is exactly the same). We examine separately  the three possible variants: $u<y$, $u>y$, and $u=y$. For more evidence we rewrite the integrand in the form
$$
\frac{(z-yq)\dots(z-yq^{N-2})}{(z-u)\dots(z-uq^{N-1})}dz.
$$
It follows that the only possible singularities are simple poles at the points $z=u, uq,\dots,uq^{N-1}$, but it may happen that some of them are annihilated  by zeros in the numerator. 

If $u<y$, then either $u=yq^m$ with $m=1,2,\dots$ or $u<0$. In both cases there are no singularities to the right of the point $yq$, so that integration over the contour $C^+(y,R)$ gives $0$. 

If $u>y$, then $u=yq^{-m}$ with $m=1,2,\dots$\,. In this case it is convenient to take the contour $C^-(y,R)$ because the integrand has no singularities to the left of the point $y$: indeed, the possible zeros of the denominator to the left of $y$ are cancelled by zeros of the numerator. 

Finally, if $u=y$, then the integrand is equal to
$$
\frac1{(z-y)(z-yq^{N-1})}.
$$
Take, for instance, the contour $C^+(y,R)$. It goes around the pole at $z=y$ in the positive direction, the pole at $z=yq^{N-1}$ lies outside, and the residue at $z=y$ equals $(1-q^{N-1})^{-1}y^{-1}$. Taking into account the prefactor $(1-q^{N-1})y$, we get the desired result. 
\end{proof}

\begin{proof}[Proof of Proposition \ref{prop3.C}]
Rewrite \eqref{eq3.A} by replacing $y$ with $u$:
$$
\sum_{u\in\L}\La^N_1(X,u)\frac1{(uz^{-1};q)_N}=\prod_{x\in X}\frac1{1-x z^{-1}}, \qquad X\in\G_N.
$$
Now multiply both sides by $(1-q^{N-1})|y|(yz^{-1}q;q)_{N-2} z^{-2}$ and integrate  over the contour $C(y)$. By virtue of Lemma \ref{lemma3.A} this gives the desired result.

Note that in the left-hand side we have to justify the interchange of summation over $u\in\L$ and integration over $z\in C(y)$, but this is easy, because the integrand can be estimated as $O(|z|^{-2})$ uniformly on $u$: here we use the fact that the measure $\La^N(X,\,\cdot\,)$ is compactly supported. 
\end{proof}

\begin{proposition}\label{prop3.D}
Formula \eqref{eq3.K} of Proposition \ref{prop3.C} can be written in the following alternate form {\rm(}as before, $X\in\G_N$ and $y\in\L${\rm):}
\begin{equation}\label{eq3.N}
\La^N_1(X,y)=\sgn(y)(1-q^{N-1})|y|\sum_{x\in X(y)}\frac{(x-yq)\dots(x-yq^{N-2})}{\prod\limits_{x'\in X\setminus\{x\}}(x-x')},
\end{equation}
where  
\begin{equation}\label{eq3.P}
X(y):=\begin{cases} \{x\in X: x\ge y\} & \text{if $y>0$,}\\   \{x\in X: x\le y\} & \text{if $y<0$.} \end{cases}
\end{equation}
\end{proposition}

\begin{proof}
As in Lemma \ref{lemma3.A}, replace in \eqref{eq3.K} the contour $C(y)$ by the closed contour $C^+(y,R)$ or $C^-(y,R)$, depending on the sign of $y$, and count the residues inside the contour. After simple transformations this gives \eqref{eq3.N}.

Note also that formula \eqref{eq3.N} is symmetric with respect to the change of sign of all variables. 
\end{proof}

Now we extend the above reasoning to the case of general $K<N$. 

\begin{theorem}\label{thm3.A}
Let $X\in\G_N$ and $Y=(y_1,\dots,y_K)\in\G_K$. We have
\begin{multline}\label{eq3.I}
\La^N_K(X,Y)=V(Y)\prod_{i=1}^K\frac{(q;q)_{N-i}}{(q;q)_{K-i}(q;q)_{N-K}}\cdot\frac{(1-q^{N-K})^K |y_1|\dots |y_K|}{(2\pi\sqrt{-1})^K}\\
\times \int_{C(y_1)}\dots\int_{C(y_K)}V(Z^{-1})\prod_{j=1}^K\frac{(y_jz_j^{-1}q;q)_{N-K-1}}{\prod_{x\in X}(1-x z_j^{-1})}\,\frac{dz_1}{z_1^2}\dots \frac{dz_K}{z_K^2}.
\end{multline}
\end{theorem}

\begin{proof}
It will be convenient to assume that $y_1<\dots<y_K$ (the right-hand side of formula \eqref{eq3.I} does not depend on the enumeration of the points in $Y$). 

We start with formula \eqref{eq3.D}, which we rewrite in the following way (below we assume that $U=(u_1<\dots<u_K)$ ranges over $\G_K$)
\begin{multline}\label{eq3.J}
\sum_{U\in\G_K}\frac{\La^N_K(X,U)}{V(U)}\det\left[\dfrac1{(u_iz_j^{-1};q)_{N-K+1}}\right]_{i,j=1}^K\\
=V(Z^{-1})\prod _{i=1}^K\frac{(q;q)_{N-i}}{(q;q)_{N-K}(q;q)_{K-i}}\cdot \prod_{x\in X}\prod_{j=1}^K\frac1{1-x z_j^{-1}}.
\end{multline}

We regard this as a generating series for the quantities $\La^N_K(X,U)/V(U)$. Given $Y=(y_1<\dots<y_K)\in\G_K$, we will extract from that series the term corresponding to $U=Y$ by making use of Lemma \ref{lemma3.A}, where we replace $N$ with $N-K+1$. To do this, we multiply both sides of \eqref{eq3.J} by
\begin{equation}\label{eq3.J1}
\prod_{j=1}^K\frac{(1-q^{N-K})}{2\pi\sqrt{-1}}|y_j|(y_jz_j^{-1}q;q)_{N-K-1}\frac{dz_j}{z_j^2}
\end{equation}
and then integrate over the contours $C(y_1),\dots,C(y_K)$. 

In the right-hand side we get the desired result, only without the factor $V(Y)$. Let us examine what will happen in the left-hand side of the equality. First, we interchange summation over $U$ and integration over $z_1,\dots,z_K$. Next, it is convenient, prior to integration, to insert the $j$th factor from \eqref{eq3.J1} into the $j$th column of the matrix under the sign of determinant and then expand the determinant into a sum of $K!$ terms indexed by permutations $\si$ of the set $\{1,\dots,K\}$. By virtue of Lemma \ref{lemma3.A}, integrating the determinant gives 
$$
\sum_{\si}\operatorname{sgn}(\si)\de_{u_{\si(1)},y_1}\dots\de_{u_{\si(K)},y_K}.
$$
Since $y_1<\dots<y_K$ and $u_1<\dots<u_K$, this equals $1$ if $U=K$ and $0$ otherwise. Thus, after integration we get in the left-hand side $\La^N_K(X,Y)/V(Y)$. This proves \eqref{eq3.I}. 
\end{proof}

\begin{remark}\label{rem3.A}
The result of Theorem \ref{thm3.A} can be easily transformed into a $K\times K$ determinantal formula by writing $V(Z^{-1})$ as a determinant. Namely,
$$
\La^N_K(X,Y)=V(Y)\prod_{i=1}^K\frac{(q;q)_{N-i}}{(q;q)_{K-i}(q;q)_{N-K}}\cdot\det[A(i,j)]_{i,j=1}^K,
$$
where
$$
A(i,j)=\frac{(1-q^{N-K})|y_j|}{2\pi \sqrt{-1}}\int_{C(y_j)}\frac{z^{i-K-2}(y_jz^{-1}q;q)_{N-K-1}}{\prod_{x\in X}(1-x z^{-1})}dz, \qquad i,j=1,\dots,K.
$$
The latter expression is similar to \eqref{eq3.K} and can be written in the following alternate form, cf. Proposition \ref{prop3.D}:
\begin{equation}\label{eq3.O}
A(i,j)=\sgn(y_j)(1-q^{N-K})|y_j|\sum_{x\in X(y)}\frac{x^{i-1}(x-yq)\dots(x-yq^{N-K-1})}{\prod\limits_{x'\in X\setminus\{x\}}(x-x')},
\end{equation}
where $X(y)\subseteq X$ is defined in \eqref{eq3.P}. Such a determinantal formula was found, for the first time, in Borodin--Olshanski \cite{BO-AdvMath} in the case of the ordinary Gelfand--Tsetlin graph. Then Petrov \cite{P} proposed a different approach, which enabled him to obtain   also a $q$-analogue of the formula. In our picture, his $q$-analogue corresponds to the case when the configurations are contained entirely in $\L^+\subset\L$. Initially, I simply repeated Petrov's computations for the whole two-sided $q$-lattice $\L$ but then I saw that one can argue somewhat differently.  
\end{remark}

The next theorem will be used in Section \ref{sect9}.  It provides a formula which extends both \eqref{eq3.D} and \eqref{eq3.I}. Fix a configuration $A=(a_1,\dots,a_m)\in\G_m$, where $0\le m\le K$, and set $n:=K-m$. Let $Z=(z_1,\dots,z_n)\in (\C\setminus\R)^n$ be an $n$-tuple of pairwise distinct numbers. We define a function $f_{A\mid Z,N,K}(Y)$ on $\G_K$ as follows (cf. \eqref{eq3.C}): 

$\bullet$ If $Y$ does not contain $A$, then $f_{A\mid Z,N,K}(Y)=0$.

$\bullet$ If $Y$ contains $A$, then denote $Y\setminus A=(y_1,\dots,y_n)$ and set
\begin{equation}\label{eq3.L}
f_{A\mid Z,N,K}(Y):=\dfrac{\det\left[\dfrac1{(y_iz_j^{-1};q)_{N-K+1}}\right]_{i,j=1}^n}{V(y_1,\dots,y_n)V(z_1^{-1}, \dots,z_n^{-1})\prod_{s=1}^n\prod_{r=1}^m(y_s-a_r)}.
\end{equation}

If $m=0$, so that $A=\varnothing$, then the function $f_{A\mid Z,N,K}(Y)$ reduces to the function $f_{Z,K,N}$, which is defined above by formula \eqref{eq3.C} and whose image under $\La^N_K$ is computed in Proposition \ref{prop3.B}. If $m$ takes the maximal possible value $m=K$, then $A\in\G_K$ and $f_{A\mid Z,N,K}(Y)$ becomes the delta-function at $A$;  the image of this delta-function under $\La^N_K$ is simply the entry $\La^N_K(X,A)$ viewed as a function in variable $X$, and this quantity is computed in Proposition \ref{prop3.A}. Now we find the image of $f_{A\mid Z,N,K}$ in the general case. 

\begin{theorem}\label{thm3.B}
Recall that $f_{A\mid Z,N,K}$ denotes the function on $\G_K$ defined in \eqref{eq3.L}. For $X\in\G_N$ we have
\begin{equation}\label{eq3.M}
\begin{gathered}
\La^N_K f_{A\mid Z,N,K}(X)=V(A)\prod_{i=1}^K\frac{(q;q)_{N-i}}{(q;q)_{K-i}(q;q)_{N-K}}\\
\times\frac{(1-q^{N-K})^m |a_1|\dots |a_m|}{(2\pi\sqrt{-1})^m}
\prod_{x\in X}\prod_{s=1}^n\frac1{1-xz_s^{-1}}  \\
\times
\int_{C(a_1)}\dots\int_{C(a_m)}V(w_1^{-1},\dots,w_m^{-1})\prod_{s=1}^n\prod_{r=1}^m(z_s^{-1}-w_r^{-1})\\
\times\prod_{r=1}^m\frac{(a_rw_r^{-1}q;q)_{N-K-1}}{\prod_{x\in X}(1-x w_r^{-1})}\,\frac{dw_1}{w_1^2}\dots \frac{dw_m}{w_m^2},
\end{gathered}
\end{equation}
where variable $w_r$ ranges over the contour $C(a_r)$, $1\le r\le m$.
\end{theorem}

In the two extreme cases, $m=0$ and $m=K$, formula \eqref{eq3.M} reduces (within notation) to \eqref{eq3.D} and \eqref{eq3.I}, respectively.

\begin{proof}
The argument is similar to that of Theorem \ref{thm3.A}, with the only difference that we have to integrate over a part of variables. Namely, we start with equality \eqref{eq3.J}, rename the last $m$ variables $z_{n+1},\dots,z_K$ into $w_1,\dots,w_m$, then multiply both sides of \eqref{eq3.J} by 
$$
\prod_{r=1}^m\frac{(1-q^{N-K})}{2\pi\sqrt{-1}}|a_r|(a_r w_r^{-1}q;q)_{N-K-1}\frac{dw_r}{w_r^2}
$$
and integrate over the contours $C(y_1),\dots,C(y_m)$. In the left-hand side of \eqref{eq3.J}, after interchanging integration over the contours with summation over $U$, we are lead to integrating the $K\times K$ determinant that enters the left-hand side. Expanding it over the last $m$ columns we see that the result of integration depends on whether  $U$ contains $A$ or  not: in the latter case we get $0$, and in the former case we get an $n\times n$ determinant of the same form as in \eqref{eq3.L}. Then the desired formula \eqref{eq3.M} appears after simple transformations. 
\end{proof}

\section{Extension of the kernels $\La^N_K$}\label{sect4}

Set $\bar\L:=\L\cup\{0\}$; this is the closure of $\L$ in $\R$.

Recall (see Section \ref{sect2.2}) that $\G_N$ (the $N$th level of the graph $\G$) can be viewed as a subset of the cone $\De_N\subset\R^N$ and denote by $\GG_N$ the closure of $\G_N$ in $\De_N$. Elements of $\GG_N$ can be described in two equivalent ways: either as $N$-point configurations on $\bar\L$ with allowed multiple points at zero or as point configurations on $\L$ of cardinality less or equal to $N$. 

According to this definition we have a stratification 
\begin{equation*}
\GG_N=\GG_{N,0}\sqcup\GG_{N,1}\sqcup\dots\sqcup\GG_{N,N},
\end{equation*}
where $\GG_{N,n}$ is formed by the configurations of the form $X=X^\circ\cup 0^n$; here $X^\circ:=X\cap\L\in\G_{N-n}$ and the symbol $0^n$ denotes $n$ points at $0$ sticked together. 

The stratum $\GG_{N,0}$ is the set $\G_N$, and the stratum $\GG_{N,N}$ consists of a sole element, $\varnothing\cup 0^N$.  We equip $\GG_N$ with the topology induced from the ambient cone $\De_N$. Then the closure of the stratum $\GG_{N,n}$ is the union of the strata $\GG_{N,n},\dots,\GG_{N,N}$.

The space $\GG_N$ is locally compact. Given two real numbers $a<b$, we denote by $\GG_N[a,b]\subset\GG_N$ the compact subset formed by the configurations contained in the closed interval $[a,b]$. If both $a$ and $b$ are nonzero, then $\GG_N[a,b]$ is open in $\GG_N$.

In the next lemma we realize\/ $\Sym(N)$ as a subalgebra of the algebra of continuous functions on the cone $\De_N\subset\R^N$. 

\begin{lemma}\label{lemma4.A}
A measure on $\De_N$ with compact support is uniquely determined by its values on the Schur polynomials $S_{\nu\mid N}\in\Sym(N)$. 
\end{lemma}

\begin{proof}
The key observation is that the functions $F\in\Sym(N)$ separate points of the cone $\De_N$. Therefore, by the Stone--Weierstrass theorem, restricting symmetric  polynomials to an arbitrary compact subset $\X\subset\De_N$ we get a dense subspace in the Banach space $C(\X)$ of continuous functions on $\X$. This implies the lemma.
\end{proof}

The set $\P(\GG_K)$ carries two topologies. One is the \emph{weak topology}, i.e. the topology of convergence on bounded continuous functions, and another is the topology of pointwise convergence of weights of atoms --- here we ignore the topology of the space $\GG_K$ and regard it simply as a countable discrete space. We need a name for the latter topology on $\P(\GG_K)$ --- let us call it the \emph{pointwise topology}. Note that it is stronger than the weak topology. 

\begin{proposition}\label{prop4.A}
Fix a pair of natural numbers $N>K$ and regard $\La^N_K$ as a map $\G_N\to\P(\G_K)$. There exists a unique extension of this map to a  map $\GG_N\to\P(\GG_K)$ which is continuous with respect to the weak topology on $\P(\GG_K)$.
\end{proposition}

\begin{proof}
This follows from Proposition \ref{prop2.C}, Proposition \ref{prop2.B},  and Lemma \ref{lemma4.A}. Indeed, assume that a sequence $\{X(n)\in\G_N\}$ converges to an element $X\in\GG_N$. Choose $a>0$ so large that the configuration $X$ is contained in $[-a,a]$. Since this set is open, the elements $X(n)$ also belong to it provided that  $n$ is large enough. 

On the other hand, by Proposition \ref{prop2.C}, the measures $\La^N_K(X(n),\, \cdot\,)$ with large $n$ are concentrated on the compact set $\GG_K[-a,a]$. Next, Proposition \ref{prop2.B} shows that they converge on the Schur polynomials. Applying Lemma \ref{lemma4.A} we see that these measures weakly converge to a probability measure on $\GG_K[a,b]$. This concludes the proof.
\end{proof}

We use the same notation $\La^N_K$ for the extended map. Since the spaces $\GG_N$ and $\GG_K$ are countable, the kernel $\La^N_K$ still may be viewed as a stochastic matrix, and we keep the same notation  $\La^N_K(X,Y)$ for its matrix entries. 

The next claims are immediate corollaries of the above argument: 
\ms

$\bullet$ The result of Proposition \ref{prop2.B} remains true for the extended matrices $\La^N_K$. That is, for every $X\in\GG_N$ and every partition $\nu$ with $\ell(\nu)\le K$
\begin{equation}\label{eq4.A}
\sum_{Y\in\GG_K}\La^N_K(X,Y)\wt S_{\nu\mid K}(Y)=\wt S_{\nu\mid N}(X).
\end{equation}
\ms

$\bullet$ The relations \eqref{eq4.A} determine $\La^N_K$ uniquely: for any $X\in\GG_N$, $\La^N_K(X,\,\cdot\,)$ is a unique probability measure on $\GG_K$ which is compactly supported and satisfies \eqref{eq4.A} for every $\nu$ with $\ell(\nu)\le K$.

\ms

$\bullet$ The extended matrices satisfy the same relations $\La^N_K=\La^N_{N-1}\dots\La^{K+1}_K$ as before.

\ms

The next lemma will be used in the proof of Theorem \ref{thm6.A} below.

\begin{lemma}\label{lemma4.B}
Let $X\in\GG_N$, where $N\ge2$, $X\ne0^N$, and let $x_0$ denote the point of $X$ with maximal absolute value, so that $x_0$ is either the leftmost or the rightmost point {\rm(}in the case these endpoints of $X$ have the same absolute value we take as $x_0$ any of them{\rm)}. 

Then the number $\La^N_1(X,x_0)$ is bounded from below by a universal positive constant{\rm:}
$$
\La^N_1(X,x_0)\ge\frac{(1-q)(q;q)_\infty}{\prod\limits_{i=0}^\infty(1+q^i)}>0.
$$
\end{lemma}

\begin{proof}
Suppose $x_0>0$. We apply  \eqref{eq3.N} and observe that for $y=x_0$ the subset $X(y)$ consists solely of $x_0$. Then \eqref{eq3.N} gives 
$$
\La^N_1(X,x_0)=\frac{(q;q)_{N-1}\,x_0^{N-1}}{\prod\limits_{x'\in X\setminus\{x_0\}}(x_0-x')}\ge\frac{(1-q)(q;q)_\infty}{\prod\limits_{x'\in X\setminus\{x_0\}}\left(1-\dfrac{x'}{x_0}\right)}.
$$
The product in the denominator is bounded from above by the convergent product $\prod_{i=0}^\infty(1+q^i)$, which gives the desired lower bound. 

In the case $x_0<0$ the argument is the same.
\end{proof}

\section{The entrance boundary: general facts}\label{sect5}

This section provides a number of definitions and known results that we will need. For more details, see \cite[Section 9]{Ols-JFA} and references therein.  We use the notation introduced in  Section \ref{sect2.1}.

Let $\VG_1,\VG_2,\dots$ be an infinite sequence of nonempty sets each of which is finite or countably infinite. For every $N$, the set $\P(\VG_N)$ can be regarded as a simplex with vertices in $\VG_N$. Assume that for every $N$ we are given a Markov kernel $\La^{N+1}_N:\VG_{N+1}\dasharrow\VG_N$. Because the spaces $\VG_N$ are discrete, these kernels are simply stochastic matrices. They determine affine maps of simplices $\P(\VG_{N+1})\to\P(\VG_N)$, so we may form the projective limit space $\varprojlim\P(\VG_N)$. In what follow we tacitly assume that the space $\varprojlim\P(\VG_N)$ is nonempty. By definition, elements of $\varprojlim\P(\VG_N)$ are sequences $\{M_N\in\P(\VG_N): N=1,2,\dots\}$ consisting of measures, which are compatible with the matrices $\Lan$ in the sense that
$$
M_{N+1}\Lan= M_N, \qquad N=1,2,\dots\,.
$$
Such sequences are called \emph{coherent families} of probability measures.  

Let $\VG$ denote the whole collection $\{\VG_N, \Lan: \, N=1,2,\dots\}$. We may regard $\VG$ as an inhomegeneous Markov chain with state spaces  $\VG_N$ and transition kernels $\Lan$, and then we define the \emph{boundary} $\partial\VG$ as the minimal entrance boundary of that chain in the sense of  Dynkin \cite{Dynkin} (the fact that our chain is space-inhomogeneous does not matter). 

An equivalent definition is that $\partial\VG$ is the set of extreme points of $\varprojlim\P(\VG_N)$. This makes sense because $\varprojlim\P(\VG_N)$ possesses an evident structure of convex set.

The space $\varprojlim\P(\VG_N)$ also possesses a natural structure of standard Borel space and the boundary $\partial\VG$ is a Borel subset.  By the very definition of projective limit, there are canonical Markov kernels
$$
\La^\infty_N: \partial\VG \dasharrow \VG_N, \qquad N=1,2,\dots,
$$
and these kernels are compatible with the matrices $\Lan$ in the sense that $$
\La^\infty_{N+1}\Lan=\La^\infty_N.
$$

The space $\varprojlim\P(\VG_N)$  is a Choquet simplex, which implies that there is a  bijective correspondence
$\P(\partial\VG)\to \varprojlim\P(\GG_N)$
given by
$$
M\mapsto \{M\La^\infty_N: N=1,2,\dots\}.
$$
In words: every coherent family of probability measures can be represented as a (continual) convex combination of extreme coherent families; this representation is unique; conversely, every (continual) convex combination of extreme coherent families is a coherent family. 

A sequence $\{X(N)\in\VG_N: N=1,2,\dots\}$ is said to be \emph{regular} if for every $K=1,2,\dots$ there exists a limit
$$
M_K=\lim_{N\to\infty}\La^N_K(X(N),\,\cdot\,)\in\P(\VG_K)
$$
with respect to pointwise topology. (In other words, for every $Y\in\VG_K$ there exists a limit $\lim_{N\to\infty} \La^N_K(X(N),Y)$ and the sum of the limit values over all $Y\in\VG_K$ equals $1$.) The sequence $\{M_K\}$ arising in this way is always a coherent family. 

For any $X\in\partial\VG$ there exists a regular sequence $\{X(N)\}$, which \emph{approximates} $X$ in the sense that the coherent family $\{M_K\}$ arising from $\{X(N)\}$ coincides with the extreme coherent family corresponding to $X$ (see Okounkov--Olshanski \cite[Theorem 6.1]{OO-IMRN}). In other words, the minimal entrance boundary is contained in the Martin entrance boundary.

\section{The $q$-boundary}\label{sect6}

Let $\GG$ denote the collection of the spaces $\GG_N$, $N=1,2,\dots$, together with the Markov kernels $\Lan:\GG_{N+1}\dasharrow\GG_N$, which were defined in Section \ref{sect4}. Our next goal is to describe the boundary $\partial\GG$.

\begin{definition} \label{def6.A}
Let $\GG_\infty$ be the set of point configurations on $\L$, which are either finite or infinite but bounded as subsets of $\R$. Given a small $\epsi>0$, we say that two configurations from $\GG_\infty$ are \emph{$\epsi$-close} to each other if they coincide outside the interval $(-\epsi,\epsi)\subset\R$. This turns  $\GG_\infty$ into a uniform space and hence a topological space.  
\end{definition}

Note the following obvious facts:

\ms

$\bullet$ For any $a>0$ the subset of $\GG_\infty$ formed by the configurations contained in the segment $[-a,a]$ is open and  compact. This in turn implies that $\GG_\infty$ is a locally compact space. 

\ms

$\bullet$ Consider the stratification 
$$
\GG_\infty=\GG_\infty^{(0)}\cup \GG_\infty^{(1)}\cup\dots\cup\GG_\infty^{(\infty)},
$$
where $\GG^{(k)}$ is the subset of $k$-point configurations and $\GG_\infty^{(\infty)}$ is the subset of infinite configurations. The closure of $\GG_\infty^{(k)}$ is the union $\GG_\infty^{(0)}\cup\dots\cup\GG_\infty^{(k)}$. Both $\GG_\infty^{(\infty)}$ and its complement are dense in $\GG_\infty$. 

\ms 

$\bullet$ There is a natural embedding $\GG_N\to\GG_\infty$: it assigns to $X=(X^\circ, 0^m)\in\GG_{N,m}$ the configuration $X^\circ\in\GG_\infty$ and establishes a bijection $\GG_{N,m}\to\GG_\infty^{(N-m)}$.

\ms

$\bullet$ Denote by $\Sym$ the algebra of symmetric functions. For $F\in\Sym$, its evaluation $F(X)$ at an arbitrary $X\in\GG_\infty$ makes sense, and $F(X)$ is a continuous function on $\GG_\infty$: this is seen from the fact that if $X$ is contained in $[-a,a]$, then there is a bound
$$
\sum_{x\in X}|x|^k\le \frac{2a^k}{1-q}, \qquad k=1,2,\dots\,.
$$
In this way one obtains a realization of $\Sym$ as a subalgebra of the algebra of continuous functions on $\GG_\infty$.  

\ms

Let  $S_\nu\in\Sym$ denote the Schur function indexed by a given partition $\nu$ and let $S_{\nu\mid\infty}$ denote the corresponding function on $\GG_\infty$. Next, by analogy with \eqref{eq2.B} we set 
$$
\wt S_{\nu\mid\infty}:=\frac{S_{\nu\mid\infty}}{S_{\nu\mid\infty}(1,q,q^2,\dots)}.
$$

\begin{theorem}\label{thm6.A}
The elements of the boundary $\partial\GG$ can be parametrized by the configurations $X\in\GG_\infty$. 

More precisely, to every $X\in\GG_N$ there corresponds a coherent family $M^{(X)}=\{M^{(X)}_K: K=1,2,\dots\}$; here  the $K$th measure $M^{(X)}_K\in\P(\GG_K)$ is uniquely determined by the relations
\begin{equation}\label{eq6.A}
\sum_{Y\in\GG_K}M^{(X)}_K(Y)\wt S_{\nu\mid K}(Y)=\wt S_{\nu\mid\infty}(X),
\end{equation}
where $\nu$ is an arbitrary partition with $\ell(\nu)\le K$. The coherent families $M^{(X)}$ are pairwise distinct and are precisely the extreme ones. 

Furthermore, the bijection $\GG_\infty\leftrightarrow\partial\GG$ is an isomorphism of Borel spaces. 
\end{theorem}

\begin{proof}
Below we denote by $\GG_\infty[-a,a]$ the subset of $\GG_\infty$ formed by the configurations contained in the closed interval $[-a,a]$, where $a>0$ is a given real number. Recall that the definition of $\GG_K[-a,a]\subset\GG_K$ is analogous, see Section \ref{sect4}. Using the embedding $\GG_K\to\GG_\infty$ defined above we may also write $\GG_K[-a,a]=\GG_K\cap\GG_\infty[-a,a]$.

\emph{Step} 1.  Let $\{X(N)\in\GG_N: N=1,2,\dots\}$ be a sequence converging to some element $X\in\GG_\infty$ (here we tacitly use the embeddings $\GG_N\to\GG_\infty$). Then $\{X(N)\}$ is regular  and the corresponding coherent family is characterized by the relations \eqref{eq6.A}; in particular, it depends only on $X$.

Indeed, choose $a>0$ so large that $X\in\GG_\infty[-a,a]$. Then $X(N)\in\GG_N[-a,a]$ for all $N$ large enough. Then, by virtue of  \eqref{eq4.A} we have
$$
\sum_{Y\in\GG_K[-a,a]}\La^N_K(X(N),Y)\wt S_{\nu\mid K}(Y)=\wt S_{\nu\mid N}(X(N)), \qquad \ell(\nu)\le K,
$$
for every large $N$ and $K<N$. Therefore, for every $K$ and any partition $\nu$ with $\ell(\nu)\le K$ there exists a limit
$$
\lim_{N\to\infty}\sum_{Y\in\GG_K[-a,a]}\La^N_K(X(N),Y)\wt S_{\nu\mid K}(Y)=\wt S_{\nu\mid\infty}(X). 
$$
Since the set $\GG_K[-a,a]$ is compact and the symmetric polynomials restricted to it are dense in the space $C(\GG_K[-a,a])$, we get the desired claim.

The above argument shows that that the coherent families $M^{(X)}$ satisfying \eqref{eq6.A} do exist. Moreover, distinct elements $X$ produce distinct families.  

\emph{Step} 2. Conversely, if a sequence $\{X(N)\in\GG_N\}$ is regular, then it converges to some element $X\in\GG_\infty$.

Indeed, Lemma \ref{lemma4.B} tells us that there exists $a>0$ such that $X(N)\in\GG_N[-a,a]$ for all $N$; otherwise the sequence $\{\La^N_1(X(N),\,\cdot\,): N=1,2,\dots\}$ of probability measures on $\GG_1=\bar\L$  fails to be tight. Then, because of compactness of $\GG_\infty[-a,a]$, choosing a subsequence of indices $N$ we may assume that $X(N)$ converges to a certain element $X\in\GG_\infty$ along this subsequence. The result of step 1 implies that $X$ does not depend on the the subsequence chosen. Therefore, $X$ is the limit of $X(N)$'s, as desired. 

\emph{Step} 3. The results of steps 1 and 2 imply that the set of extreme coherent families is contained in the set of coherent families of the form $M^{(X)}$ with $X\in\GG_\infty$. We want to prove that both sets actually coincide. 

Let us fix an arbitrary element $X\in\GG_\infty$ and show that $M^{(X)}$ is extreme. We know that  $M^{(X)}$ can be represented, in a unique way, as a (continual) convex combination of extreme coherent families, governed by a mixing probability Borel measure $\pi$ supported by the set $\partial\GG$:
$$
M^{(X)}=\int_{M'\in\partial\GG} M' \pi(dM').
$$

Since each $M'$ has the form $M^{(X')}$ for some $X'\in\GG_\infty$, we would like to interpret $\pi$ as a Borel measure on $\GG_\infty$. But here is a subtle point: we need to know that the natural Borel structure on the space $\GG_\infty$ coincides with that induced from the ambient space $\varprojlim\P(\GG_N)$. But this is indeed true, as shown in the final step 4. Using this fact we may write
$$
M^{(X)}=\int_{X'\in\GG_\infty}M^{(X')}\pi(dX'), 
$$
and then we have to prove that $\pi$ is actually the delta measure at $X$.

The exact meaning of the above equality is that
$$
M^{(X)}_K(Y)=\int_{X'\in\GG_\infty}M^{(X')}_K(Y)\pi(dX')
$$
for every $K=1,2,\dots$ and any $Y\in\GG_K$. Setting $K=1$ and applying Lemma \ref{lemma4.B} we see that $\pi$ is concentrated on a compact subset of the form $\GG_\infty[-a,a]$. Then we may apply \eqref{eq6.A}, which implies that 
$$
\wt S_{\nu\mid\infty}(X)=\int_{X'\in\GG_\infty[-a,a]}\wt S_{\nu\mid\infty}(X')\pi(dX')
$$
for any partition $\nu$. But since the symmetric functions are dense in the Banach space $C(\GG_\infty[-a,a])$, this may happen only if $\pi$ is the delta measure at $X$, as desired. 

\emph{Step} 4. 
It remains to justify the translation of $\pi$ to $\GG_\infty$. That is, we have to prove that the injective map  $\GG_\infty \to \varprojlim\P(\GG_N)$ assigning to $X\in\GG_\infty$ the corresponding coherent family $M^{(X)}$ is a Borel isomorphism onto its image. A way to do this is to apply a theorem from descriptive set theory (see Kechris \cite[Corollary 15.2]{Kechris}). The hypotheses of this theorem are satisfied because both spaces are standard and the map is Borel. 

This completes the proof.
\end{proof}

\begin{corollary}\label{cor6.A}
Let $X\in\GG_\infty$ and $Y=(y_1,\dots,y_K)\in\G_K$.  We have
\begin{multline}\label{eq6.B}
\La^\infty_K(X,Y)=V(Y)\prod_{i=1}^K\frac1{(q;q)_{K-i}}\cdot\frac{|y_1|\dots |y_K|}{(2\pi\sqrt{-1})^K}\\
\times \int_{C(y_1)}\dots\int_{C(y_K)}V(Z^{-1})\prod_{j=1}^K\frac{(y_jz_j^{-1}q;q)_\infty}{\prod_{x\in X}(1-x z_j^{-1})}\,\frac{dz_1}{z_1^2}\dots \frac{dz_K}{z_K^2}.
\end{multline}
\end{corollary}

\begin{proof}
Choose a sequence $\{X(N)\in\G_N\}$  converging to $X$. The argument of Step 1 in the proof of Theorem \ref{thm6.A} shows that $\La^\infty_K(X,Y)=\lim_{N\to\infty}\La^N_K(X(N),Y)$. Then we pass to the limit in \eqref{eq3.I}.
\end{proof}

\begin{remark}\label{rem6.A}
Formula \eqref{eq6.B} looses its meaning if $Y\in\GG_K\setminus\G_K$, because then the $K$-fold integral in the right-hand side is not defined.
\end{remark}

\section{Divided differences,  $q$-B-splines, and division by the Vandermonde}\label{sect7}

Let us recall a few facts about divided differences and B-splines. For more details, see de Boor \cite{deBoor}, Curry and Schoenberg \cite{CS}, Faraut \cite{Faraut-Peano}, \cite{Faraut}, Phillips \cite{Phillips}.

Given a function $f(t)$, its \emph{divided difference  with $N$ pairwise distinct knots $x_1,\dots,x_N$} is the quantity $f[x_1,\dots,x_N]$ defined recursively by
$$  
\begin{gathered}
f[x_1]:=f(x_1), \qquad f[x_1,x_2]:=\frac{f(x_2)-f(x_1)}{x_2-x_1},\; \dots\\
\dots\;, f[x_1,\dots,x_N]:=\frac{f[x_2,\dots,x_N]-f[x_1,\dots,x_N]}{x_N-x_1}.
\end{gathered}
$$
The quantity $f[x_1,\dots,x_N]$ is invariant under permutations of $x_i$'s. We have
\begin{equation}\label{eq7.G}
\begin{gathered}
t^n[x_1,\dots,x_N]=0, \qquad n<N-1\\
t^{m+N-1}[x_1,\dots,x_N]=h_m(x_1,\dots,x_N), \qquad m=0,1,2,\dots,
\end{gathered}
\end{equation}
where $h_m\in\Sym$ is the complete homogeneous symmetric function of degree $m$.

Suppose now that $x_1,\dots,x_N$ are real numbers. There exists a unique probability measure $B_N(x_1,\dots,x_N)$ on $\R$ with the moments
\begin{equation}\label{eq7.A}
\langle t^m, B_N(x_1,\dots,x_N) \rangle = \frac{m!(N-1)!}{(m+N-1)!}h_m(x_1,\dots,x_N),
\end{equation}
where the angular brackets denote the pairing between functions and measures. The support of $B_N(x_1,\dots,x_N)$ is the smallest closed interval containing the knots $x_1,\dots,x_N$. If the knots are pairwise distinct, then $B_N(x_1,\dots,x_N)$ has a density $B_N(x_1,\dots,x_N; t)$, which is of class $C^{N-3}$ and is given by a polynomial of degree $N-2$ on each interval between two consecutive knots. 

The measure $B_N(x_1,\dots,x_N)$ (or its density $B_N(x_1,\dots,x_N;t)$) is called the \emph{B-spline}. Initially it was called ``fundamental spline'' (Curry and Schoenberg \cite{CS}). Some authors use a different normalization: so, in Phillips' book  \cite{Phillips} the term B-spline refers to the function $\frac{x_N-x_1}{N-1} B_N(x_1,\dots,x_N;t)$, where it is supposed that $x_1<\dots<x_N$ (but then the symmetry in $x_1,\dots,x_N$ is lost). 

The B-spline is linked to divided differences by the \emph{Hermite--Genocchi formula}:  if $f(t)$ is a function on $\R$ of class $C^{N-1}$, then 
\begin{equation}\label{eq7.B}
f[x_1,\dots,x_N]=\frac1{(N-1)!}\int_\R f^{(N-1)}(t) B_N(x_1,\dots,x_N;t)dt.
\end{equation}

As shown in recent papers \cite{SG}, \cite{BDGO},  for the B-spline and various related classic formulas including the Hermite--Genocchi formula, there exist $q$-analogues.  I will state below a few results from these works with minor modifications and a different proof.  

Recall the definitions of $q$-derivation and $q$-integration (see, e.g. Gasper and Rahman \cite{GR}); both are well adapted to the lattice $\L$.

Given a function $f(t)$ on $\L$, its \emph{$q$-derivative} is defined by
$$
(D_q f)(t)=\frac{f(t)-f(tq)}{t(1-q)}, \qquad t\in\L.
$$
In what follows we drop the parentheses and write $D_q f(t)$ instead of $(D_q f)(t)$.

The operator $D_q$ preserves the space of polynomials: we have
$$
D_q t^n =[n]_q t^{n-1},
$$
where
$$
[n]_q:=\frac{1-q^n}{1-q}.
$$

\begin{definition}\label{def7.A}
Denote by $C^0(\L)$ the space of functions on $\bar\L$ that are continuous at $0$. Next, for $n=1,2,\dots$,  let $C^n(\L)$ be the subspace of functions $f\in C^0(\L)$ such that for every $\ell=1,\dots,n$ there exists a limit
$$
D^\ell_qf(0):=\lim_{t\to 0}D^\ell_q f(t),
$$
where $D^\ell_q$ is the $\ell$th power of $D_q$. For $f\in C^n(\L)$ we also say that it is of class $C^n$. Finally,   $f\in C^\infty(\L)$ means that $f$ is of class $C^n $ for all $n$. For instance, if $f$ coincides with a polynomial in a neighborhood of $0$, then $f\in C^\infty(\L)$. 
\end{definition}

The \emph{$q$-integral} can be defined as follows. The \emph{canonical measure} $\mu$ on $\L$ is the infinite measure on $\L$ with weights $\mu(t)=(1-q)|t|$. If $f(t)$ is of class $C^0$ and $a<b$ are two points of $\L$, then we set
$$
\int_a^b f(t) d_qt=-\int_b^a f(t) d_qt=\langle f, \mu_{I(a,b)}\rangle,
$$
where the interval $I(a,b)$ is defined in \eqref{eq2.D} and  $\mu_{I(a,b)}$ denotes the restriction of $\mu$ to $I(a,b)$. The definition trivially extends to the case when $a$ or $b$ equals $0$. 

Two basic formulas of $q$-calculus are analogues of Newton--Leibniz and integration by parts:
\begin{equation}\label{eq7.C}
f(b)-f(a)=\int_a^b f(t) d_qt, 
\end{equation}
where $f$ is of class $C^1$, and
\begin{equation}\label{eq7.D}
\int D_q f(t) g(t) d_qt=-\int f(tq) D_q g(t) d_qt,
\end{equation}
where $f$ and $g$ are of class $C^1$ and one of them has bounded support (the $q$-integral without limits is understood as the integral against $\mu$). 

Below we denote by $[m]_q! $ the $q$-factorial:
$$
[m]_q!:=[1]_q\dots[m]_q=\frac{(q;q)_m}{(1-q)^m}, \qquad m=0,1,2,\dots\,.
$$
If $X=(x_1,\dots,x_N)\in\G_N$, then $h_m(X):=h_m(x_1,\dots,x_N)$. More generally, if $X\in\GG_N$, then the value of the function $h_m\in\Sym$ at $X$ is defined by continuity; thus, if $X=(x_1,\dots,x_k)\cup 0^{N-k}\in\GG_N$, where $x_1,\dots,x_k\in\L$ are pairwise distinct, then $h_m(X)=h_m(x_1,\dots,x_k)$. 

\begin{theorem}\label{thm7.A}
{\rm(i)} Let $X\in\GG_N$ be arbitrary. There exists a unique probability measure $B^q_N(X)$ on $\bar\L$ with the moments
\begin{equation}\label{eq7.E}
\langle t^m, B^q_N(X)\rangle=\frac{[m]_q![N-1]_q!}{[m+N-1]_q!}h_m(X), \qquad m=0,1,2,\dots\,.
\end{equation}
Its support is contained in the smallest segment containing $X$.

{\rm(ii)} If $X=(x_1,\dots,x_N)\in\G_N$ and $f\in C^{N-1}(\L)$, then
\begin{equation}\label{eq7.F}
f[x_1,\dots,x_N]=\frac1{[N-1]_q!}\langle D_q^{N-1}f, B^q_N(X)\rangle.
\end{equation}
\end{theorem}

\begin{proof}
(i) The uniqueness claim is trivial, because the moments do not grow too fast. The existence follows from \eqref{eq3.B}: we simply set $B^q_N(X):=\La^N_1(X,\,\cdot\,)$. As was already pointed out in the proof of Proposition \ref{prop3.A}, we have 
$$
h_m(1,q,\dots,q^{N-1})=\frac{[m+N-1]_q!}{[m]_q![N-1]_q!}.
$$
Together with \eqref{eq3.B} this shows that $B^q_N(X)$ has the required moments. The claim concerning the support is evident.

(ii) Formula \eqref{eq7.F} holds true in the case when $f(t)$ is a monomial: this is easily derived from \eqref{eq7.E} and \eqref{eq7.G}. The general case is reduced that case: we write 
$$
f(t)=\sum_{n=0}^{N-2}\frac{D_q^n f(0)}{[n]_q!}t^n+\int_0^t d_qt_1\int_0^{t_1}d_q t_2\dots\int_0^{t_{N-2}}D_q^{N-1} f(t_{N-1})d_q t_{N-1}
$$
and then approximate the function $D_q^{N-1} f(t)$, on a sufficiently large interval $[-a,a]\cap\bar\L$, by polynomials. 
\end{proof}

\begin{definition}\label{def7.B}
We call the measure $B^q_N(X)$ the \emph{$q$-B-spline} and the formula \eqref{eq7.F}, the \emph{$q$-Hermite-Genocchi formula}.  
\end{definition}

\begin{corollary}\label{cor7.A}
Fix an arbitrary  $f\in C^{N-1}(\L)$. The function $X=(x_1,\dots,x_N)\mapsto f[x_1,\dots,x_N]$, initially defined on $\G_N$, admits a continuous extension to the space $\GG_N\supset\G_N$ given by the $q$-Hermite--Genocchi formula \eqref{eq7.F}. In particular,
\begin{equation}\label{eq7.K}
\lim_{x_1,\dots,x_N\to0}f[x_1,\dots,x_N]=\frac{D_q^{N-1} f(0)}{[N-1]_q!}.
\end{equation}
\end{corollary} 

\begin{proof}
With no loss of generality we may consider only configurations $X$ contained in a fixed interval $[-a,a]\subset\R$. But then  the measure $B^q_N(X)$ is concentrated on $[-a,a]$, too. It depends continuously on $X$, because the moments are continuous functions on $\GG_N$. Then the existence of the continuous extension follows from the $q$-Hermite--Genocchi formula \eqref{eq7.F}. The limit relation \eqref{eq7.K} follows from the fact that, as the knots converge to $0$, the $q$-B-spline converges to the delta measure at $0$.
\end{proof}

Let $f_1,\dots,f_n$ be functions on $\L$. For  $X=(x_1,\dots,x_n)\in\G_n$ we set
\begin{equation}\label{eq7.H}
F(X)=F(x_1,\dots,x_n):=\frac{\det\left[f_j(x_i)\right]_{i,j=1}^n}{V(x_1,\dots,x_n)}.
\end{equation}
The definition is correct, because $F(x_1,\dots,x_n)$ is symmetric with respect to permutations of the arguments. 

\begin{corollary}\label{cor7.B}
Suppose $f_1,\dots,f_n\in C^{n-1}(\L)$. Then the function \eqref{eq7.H}, initially defined on $\G_n$, admits a continuous extension to $\GG_n\supset\G_n$ given by
\begin{equation}\label{eq7.L}
F(x_1,\dots,x_n)=\frac{(-1)^{n(n-1)/2}}{\prod_{\ell=1}^n(\ell-1)!}\cdot\det\left[\langle D_q^{\ell-1}f_j,\, B^q_\ell(x_1,\dots,x_\ell)\rangle\right]_{\ell,j=1}^n.
\end{equation}
In particular,
\begin{equation}\label{eq7.F1}
\lim_{x_1,\dots,x_n\to0}F(x_1,\dots,x_n)=\frac{(-1)^{n(n-1)/2}}{\prod_{\ell=1}^n(\ell-1)!}\cdot\det [D_q^{(\ell-1)}f_j(0)]_{\ell,j=1}^n.
\end{equation}
\end{corollary}

\begin{proof}
Observe that
\begin{equation}\label{eq7.I}
\det[f_j(x_i)]_{i,j=1}^n=(-1)^{n(n-1)/2}V(x_1,\dots,x_n)\det[f_j[x_1,\dots,x_\ell]]_{\ell,j=1}^n.
\end{equation}
Indeed, to see this we perform the following elementary transformations under the rows of the matrix  rows $[f_j(x_i)]_{i,j=1}^n$. On the first step, we subtract from the $i$th row the $(i-1)$th one, starting with $i=n$ and ending with $i=2$. On the second step we iterate the procedure, starting from $i=n$ and ending with $i=3$, and so on. 

By virtue of \eqref{eq7.I}, we have 
\begin{equation}\label{eq7.J}
F(x_1,\dots,x_n)=(-1)^{n(n-1)/2}\det[f_j[x_1,\dots,x_\ell]]_{\ell,j=1}^n.
\end{equation}
Now \eqref{eq7.L} and \eqref{eq7.F1}  follow from Corollary \ref{cor7.A}, where we replace $N$ by $\ell$. 
\end{proof}

\section{Total systems of vectors in $C_0(\GG_K)$}\label{sect8}

In this section $K$ is a fixed positive integer. We denote by $C_0(\GG_K)$ the space of real-valued continuous functions on $\GG_K$ vanishing at infinity. This is a separable Banach space with respect to the supremum norm.

Recall (see Section \ref{sect3}) that $\GG_{K,n}$ denotes the set of configurations $X\in\GG_K$ containing $0$ with multiplicity $n$. We also write
$$
\GG_{K,\ge n}:=\GG_{K,n}\cup\dots\cup\GG_{K,K}.
$$
This is a closed subset of $\GG_K$.

Given a finite subset $A\subset\L$, we denote by $\card(A)$ its cardinality; throughout this section we consider only subsets $A$ with $\card(A)\le K$.  By $\GG_K(A)$ we denote the set of configurations $X\in\GG_K$ containing $A$.  It is both open and closed in $\GG_K$. Its intersection with $\GG_{K,n}$ is nonempty if and only if $n\le K-\card(A)$, and the intersection with $\GG_{K,K-\card(A)}$ consists of a single element, which we denote by $X_A$:  
$$
X_A:=0^{K-\card(A)}\cup A.
$$

If $\card(A)=K$, then the set $\GG_K(A)$ is the singleton $\{A\}$ and $X_A=A$. If $A$ is empty, then $\GG_K(\varnothing)$ is the whole space $\GG_K$ and $X_{\varnothing}=0^K$. 

In the next proposition we assume that for each subset  $A\subset\GG_K$ with $\card(A)\le K$ we a given a function $f_A\in C_0(\GG_K)$ supported by $\GG_K(A)$ and such that $f_A(X_A)\ne0$. In particular, if $\card(A)=K$, then $f_A$ is proportional to the  delta function at $A$, and if $A=\varnothing$, then the only condition on $f_{\varnothing}$ is that it belongs to $C_0(\GG_K)$ and takes a nonzero value at $0^K$. 

As will be shown later, examples are provided by functions of the form \eqref{eq7.H}.  

\begin{proposition}\label{prop8.A}
For an arbitrary choice of the functions $f_A$ as indicated above, they form a total system in $C_0(\GG_K)$, i.e. their linear span is dense in the norm topology.  
\end{proposition}

\begin{proof}
Let $\V$ be the shorthand notation for the Banach space $C_0(\GG_K)$. We denote by $\V^m$ the subspace of $\V$ formed by the functions vanishing on $\GG_{K,\ge m}$. Evidently,
$$
\V\supset\V^K\supset\dots\supset\V^0=\{0\}.
$$

Next, let $\W$ denote the linear span of the functions $f_A$. Observe that 
\begin{equation}\label{eq8.A}
\V=\W+\V^K.
\end{equation}
Indeed, $\V^K\subset\V$ consists of the functions vanishing at $0^K$. If $f\in\V$ is arbitrary, then 
$$
f=f(0^K)f_{\varnothing}+(f-f(0^K)f_{\varnothing})\in\W+\V^K.
$$

We are going to prove the following statement: for any $n$, $1\le n\le K$, one has 
\begin{equation}\label{eq8.B}
\V^n\subseteq\overline{\W+\V^{n-1}},
\end{equation}
where the bar means closure.

Once \eqref{eq8.B} is established, we immediately get the desired equality $\V=\overline{\W}$, because we may write
$$
\V=\W+\V^K\subseteq\overline{\W+\V^{K-1}}\subseteq\dots\subseteq\overline{\W+\V^0}=\overline{\W}.
$$
Here the first equality is \eqref{eq8.A} and each inclusion is justified by \eqref{eq8.B}. 

We proceed now to the proof of \eqref{eq8.B}. From now on and till the end of the proof $n$ is fixed and $\mathscr A$ denotes the set of all subsets $A\subset\L$ with $\card(A)=K-n+1$. We have a disjoint union decomposition 
$$
\GG_K=\GG_{K,\ge n}\sqcup\left(\GG_{K,n-1}\sqcup\dots\sqcup\GG_{K,0}\right),
$$
where the set in the parentheses coincides with $\bigcup_{A\in\mathscr A}\GG_K(A)$. Therefore,
$$
\GG_K=\GG_{K,\ge n}\sqcup \bigcup_{A\in\mathscr A}\GG_K(A).
$$

Fix an arbitrary function $f\in\V^n$ and show that it can be approximated by functions from $\W+\V^{n-1}$. 
This is done in three steps.

1. Recall that $\GG_K[-a,a]$ denotes the set of configurations contained in $[-a,a]\subset\R$; this is an open compact set in $\GG_K$. Since $f$ vanishes at infinity, we have
$$
\lim_{a\to+\infty}\sup_{X\notin \GG_K[-a,a]}|f(X)|=0.
$$
Therefore, without loss of generality, we may assume that $f$ vanishes outside $\GG_K[-a,a]$ for some $a>0$.

2. The set $\GG_K[-a,a]$ is a compact uniform space, hence $f$ is uniformly continuous on it. Since $f$ vanishes on $\GG_{K,\ge n}\cap\GG_K[-a,a]$, it follows, that given small $\de>0$, there exists a small $\epsi>0$ such that $|f(X)|<\de$ whenever $X$ contains at least $n$ points in the interval $[-\epsi,\epsi]$.  In other words, $|f(X)|\ge\de$ implies that $X$ has at least $K-n+1$ points outside $[-\epsi,\epsi]$, which in turn means that $X$ belongs to the union of sets $\GG_K(A)$ such that $A\subset[-a,a]\setminus[-\epsi,\epsi]$. But there are finitely many such $A$'s. We conclude that there exists a finite subset $\mathscr A_0\subset\mathscr A$ such that $|f(X)|<\de$ outside $\bigcup_{A\in\mathscr A_0}\GG_K(A)$. 

3. The above  argument makes it possible to further reduce the problem to the case when $f$ is supported by a set of the form $\bigcup_{A\in\mathscr A_0}\GG_K(A)$, where $\mathscr A_0$ is finite (here we use the fact that any such set is both open and closed in $\GG_K$). Now we write $f$ as a sum of two components, $ f=f'+f''$, where 
$$
f'=\left(\sum_{A\in\mathscr A_0}f(X_A)f_A\right)
$$
and $f'':=f-f'$. Obviously, $f'\in\W$. As for $f''$, it belongs to $\V^{n-1}$. Indeed, to see this we observe that the intersection $\GG_K(A)\cap\GG_{K,n-1}$ consists of the single element $X_A$, which also implies that $f_A(X_{A'})=0$ if $A\ne A'$; it follows that $f''(X_A)=0$ for every $A\in\mathscr A_0$, so that $f''$ vanishes on $\GG_{K,n-1}$, which means that $f''\in\V^{n-1}$. 

This completes the proof. 
\end{proof}

\section{The Feller property}\label{sect9}

Let $\X$ and $\Y$ be two locally compact (but noncompact) spaces, and $C_0(\X)$ and $C_0(\Y)$ be the corrresponding Banach spaces of continuous functions vanishing at infinity. A Markov kernel $\La:\X\dasharrow \Y$ is said to be \emph{Feller} if the corresponding contraction operator maps $C_0(\Y)$ to $C_0(\X)$. 

\begin{theorem}\label{thm9.A}
The kernels $\La^N_K: \GG_N\dasharrow\GG_K$, where $N>K\ge1$, are Feller.
\end{theorem} 

For the proof we need two lemmas.

Let $n=2,3,\dots$ and $f_1,\dots,f_n\in C^{n-1}(\L)$. By Corollary \ref{cor7.B}, the function
\begin{equation}\label{eq9.A}
F(Y)=F(y_1,\dots,y_n):=\frac{\det[f_j(x_i)]_{i,j=1}^n}{V(y_1,\dots,y_n)},
\end{equation}
initially defined on $\G_n$, admits a continuous extension to \/ $\GG_n$. Below we keep the same notation $F(Y)$ or $F(y_1,\dots,y_n)$ for the resulting function on $\GG_n$.

\begin{lemma}\label{lemma9.A}
Assume additionally that the functions $f_1,\dots,f_n$ and all their $q$-derivatives up to order $n-1$ are bounded on $\bar\L$.  Then $F$ vanishes at infinity. 
\end{lemma}

\begin{proof}
It is convenient to enumerate the points $y_1,\dots, y_n$ in increasing order. Then $Y\to\infty$ means that at least one of the following conditions holds:  $y_n\to+\infty$ or $y_1\to-\infty$. Let, for definiteness, $y_n\to+\infty$. Then we write $F$ in the form 
$$
F(Y)=\frac{\det[f_j(y_i)]_{i,j=1}^n}{V(y_1,\dots,y_{n-1})}\cdot\frac1{\prod_{i=1}^{n-1}(y_i-y_n)}.
$$
The first fraction is bounded on $\GG_N$: to see this we expand determinant in the numerator on the last  row, apply Corollary \ref{cor7.B}, and use hypotheses about functions $f_1,\dots,f_n$. As for the second fraction, it goes to $0$ as $y_n\to+\infty$, because  $|y_i-y_n|\ge(1-q) y_n$ for any $i=1,\dots,n-1$ (here it is essential that $n\ge2$, otherwise the product would be empty). 
\end{proof}

\begin{lemma}\label{lemma9.C}
Fix $N>K$ and consider the functions
\begin{equation}\label{eq9.B}
f_{j}(y)=\frac1{(yz_j^{-1};q)_{N-K+1}}, \qquad y\in\L, \quad j=1,\dots,n,
\end{equation}
where $z_1,\dots,z_n\in\C\setminus\R$.

{\rm(i)} These functions satisfy the hypotheses of Lemma \ref{lemma9.A}.

{\rm(ii)} Let $F(y_1,\dots,y_n)$ be the corresponding function in $n$ variables, defined by \eqref{eq9.A}. If the numbers $z_1,\dots,z_n$ are pairwise distinct, then $F(0^n)\ne0$.
\end{lemma}

\begin{proof}
(i) Observe that
\begin{equation}\label{eq9.C}
D_q \left\{\frac1{(yz^{-1};q)_M}\right\}=\frac{[M]_q z^{-1}}{(yz^{-1};q)_{M+1}}, \qquad M=1,2,\dots,
\end{equation}
where the $q$-derivative is taken with respect to variable $y$. Furthermore, the quantity $|(yz^{-1};q)_M|$ tends to infinity as $|y|\to\infty$. It follows that our functions lie in $C^\infty(\L)$ and are uniformly bounded together with all their $q$-derivatives. This proves (i).

(ii) By virtue of \eqref{eq7.F1}, the quantity $F(0^n)$ is equal, within a nonzero number factor, to the determinant $\det[D_q^{\ell-1}f_{j}(0)]_{\ell,j=1}^n$.  Iterating \eqref{eq9.C}  we see that this  determinant is equal, within a nonzero scalar factor, to $V(z_1^{-1},\dots,z_n^{-1})$ and hence does not vanish.
\end{proof}

\begin{proof}[Proof of Theorem \ref{thm9.A}]

The idea of the proof is the following. In Section \ref{sect3} we computed the action of $\La^N_K$ on certain functions. Using Proposition \ref{prop8.A} and the lemmas given above we will show that these functions form a total family in $C_0(\GG_K)$. On the other hand, from the formulas of Section \ref{sect3} it is seen that the images of our functions under the action of $\La^N_K$ lie in $C_0(\GG_N)$. Because the operator with kernel $\La^N_K$ is contractive, it follows that it maps the whole space $C_0(\GG_K)$ into $C_0(\GG_N)$. 

Consider the  functions $f_{A\mid Z,N,K}$ defined in \eqref{eq3.L}. From Theorem \ref{thm3.B} it follows that $\La^N_K f_{A\mid Z,N,K}\in C_0(\GG_N)$. Recall that  $A$ is an arbitrary subset of $\L$ of cardinality $\card A\le K$. We set $m:=\card A$ and $n:=K-m$. It is not necessary to consider all possible $Z$'s; for our purpose it suffices to pick, for every $A$, some $n$-tuple $Z=(z_1,\dots,z_n)$ of pairwise distinct numbers from $\C\setminus\R$. Then we set  $f_A:=f_{A\mid Z,N,K}$. 

We are going to show that the family $\{f_A\}$ obtained in this way satisfies the two hypotheses of Proposition \ref{prop8.A}: namely, $f_A\in C_0(\GG_K)$ and $f_A(0^n\cup A)\ne0$ (when $n\ge1$). This will imply that $\{f_A\}$ is a total family. 

Examine first the simplest case $K=1$. Then either $A=\varnothing$ or $A=\{a\}$, where $a\in\L$. 

If $A=\varnothing$, then $f_A$ is the function $f_\varnothing(y)=\dfrac1{(yz^{-1};q)_N}$. It is evidently in $C_0(\GG_1)$ and does not vanish at $y=0$. 

If $A=\{a\}$, then $n=0$ (hence the second condition disappears) and $f_A$ is the delta function at $a$. Again, it is evidently in $C_0(\GG_1)$.

Let us proceed to the case $K\ge2$. Then $f_A$ is given by formula \eqref{eq3.L}. 

Let us check that $f_A\in C_0(\GG_K)$. The double product in the denominator in the right-hand side of \eqref{eq3.L} causes no problem, and we may ignore it. We may also ignore $V(z_1^{-1},\dots,z_n^{-1})$, which is a nonzero constant. Then we are left with an expression of the form \eqref{eq9.A}. It suffices to show that it depends continuously on variables $y_1,\dots,y_n$ and vanishes at infinity, but this follows from Lemma \ref{lemma9.A} and Lemma \ref{lemma9.C} (i). 

It remains to check that $f_A(0^n\cup A)\ne0$, but this follows from Lemma \ref{lemma9.C} (ii). 
\end{proof}

Our aim is to extend Theorem \ref{thm9.A} to the kernels $\La^\infty_K: \GG^\infty\dasharrow\GG_K$. For this purpose we need one more lemma, where we are dealing with an infinite sequence of $n$-tuples $\{f_{1,N},\dots,f_{n,N}\}$ depending on an index $N$. We suppose that the following conditions hold:

$\bullet$ for any fixed $N$, the corresponding $n$-tuple satisfies the hypotheses of Lemma \ref{lemma9.A};

$\bullet$ as $N\to\infty$, there exist uniform limits $f_{j,N}\to f_{j,\infty}$, and the same holds for all  $q$-derivatives up to order $n-1$. 

We denote by $F_N$ and $F_\infty$ the corresponding functions on $\GG_n$.

\begin{lemma}\label{lemma9.B}
Under these assumptions, $F_N\to F_\infty$ uniformly on $\GG_n$. 
\end{lemma}

\begin{proof}
This immediately follows from Corollary \ref{cor7.B}.
\end{proof}

Note that the hypotheses of Lemma \ref{lemma9.B} are satisfied for the functions
\begin{equation}
f_{j,N}(y):=\dfrac1{(yz_j^{-1};q)_{N-K+1}}, \qquad f_{j,\infty}(y)=\dfrac1{(yz_j^{-1};q)_\infty}.
\end{equation}
Below we apply the lemma with these concrete functions. 

\begin{theorem}\label{thm9.B}
The kernels $\La^\infty_K: \GG_\infty\dasharrow \GG_K$, $K=1,2,\dots$, are Feller.
\end{theorem}

\begin{proof}
Observe that all claims and formulas in Section \ref{sect3} have evident analogues with the kernels $\La^\infty_K$ replacing the kernels $\La^N_K$. Indeed, given $X\in\GG_\infty$, we approximate it by a sequence $X(N)\in\GG_N$ and then pass to the limit as $N\to\infty$. The limit transition is justified by using Lemma \ref{lemma9.B} and the fact that the measures $\La^N(X(N),\,\cdot\,)$ weakly converge to the measure $\La^\infty_K(X,\,\,\cdot\,)$. After this we repeat the same argument.
\end{proof}

Note that, in the context of the graph $\GT$ and its $q$-boundary, the $N=\infty$ version of Proposition  \ref{prop3.B} was earlier obtained by Gorin in a different way (see claim 2 of Theorem 1.1 in \cite{Gorin}).

\section{Concluding remarks}\label{sect10}

\subsection{Hierarchy of splines}
Besides $q$-B-splines there exists another discrete analogue of the classical B-splines, the so-called \emph{$h$-B-splines}. They arise when the $q$-lattice is replaced by the ordinary lattice $\Z$. Letting $q\to1$ and focusing on a small neighborhood of the point $\zeta_+\in\L$, one can degenerate the $q$-B-splines into the $h$-B-splines (Simeonov and Goldman, \cite[Appendix]{SG}). Further, in a natural scaling limit the $h$-B-splines degenerate into the classical B-splines. 

On the other hand, one can directly degenerate the $q$-B-splines into the classical B-splines.

\subsection{Hierarchy of branching graphs}
The hierarchy of splines mentioned above corresponds to the following hierarchy of branching graphs: the top position is occupied by the graph $\G$, the classical Gelfand--Tsetlin graph is in the middle, and the object at the bottom is a continuous analogue of $\GT$. The latter object is not a graph in the strict sense, because its levels are continuous. Kerov and I called it the ``graph of spectra'', as it describes the branching of eigenvalues of Hermitian matrices. 

In all three cases, the Markov kernels $\La^N_1$ are given by the corresponding versions of splines ($q$-B-splines, $h$-B-splines, and conventional B-splines, respectively). More generally, in all three cases there are determinantal formulas for the more general kernels $\La^N_K$: see respectively Remark \ref{rem3.A};  Borodin--Olshanski \cite{BO-AdvMath}; Olshanski \cite{Ols-JLT} and Faraut \cite[Theorem 6.2]{Faraut}.

\subsection{The work of Curry and Schoenberg \cite{CS}}
In Theorem 6 of \cite{CS} (see also the announce in \cite{CS1947}), Curry and Schoenberg described all possible limits of the classical B-splines as $N$ (the number of knots) goes to infinity.  They discovered that the answer is the same as in the problem of classification of totally positive functions, which was solved by Schoenberg \cite{S1947}, \cite{S}. In our understanding,  the problem investigated by Curry and Schoenberg is a part of the problem of describing the boundary of the graph of spectra (see Olshanski--Vershik \cite[Section 8]{OV}).  

An analogue of the Curry--Schoenberg result also holds for the $q$-B-splines on $\bar\L$ and the $h$-B-splines on $\Z$, and the limiting objects are again parameterized by the points of the boundary of the corresponding graph, i.e. $\GG$ and $\GT$, respectively.

\subsection{$q$-Laplace transform}
Given $N\ge2$ and a (complex) measure $M$ on $\L$, define its transform $\mathbb L_N M$ as the function of complex variable $z$, given by
\begin{equation}\label{eq3.Q1}
\varphi(z)=(\mathbb L_N M)(z):=\sum_{y\in\L}\frac1{(yz^{-1};q)_N} M(y).
\end{equation}

Lemma \ref{lemma3.A} says that the inverse transform is given by
\begin{equation}\label{eq3.Q2}
M(y)=(\mathbb L_N^{-1}\varphi)(y)=\frac{(1-q^{N-1})|y|}{2\pi\sqrt{-1}}\int_{C(y)}(yz^{-1}q;q)_{N-2}\,\varphi(z)\frac{dz}{z^2}.
\end{equation}

Likewise, in the limit as $N\to\infty$, we obtain two mutually inverse transforms
\begin{equation}\label{eq3.Q3}
\varphi(z)=(\mathbb L M)(z):=\sum_{y\in\L}\frac1{(yz^{-1};q)_\infty} M(y).
\end{equation}
and
\begin{equation}\label{eq3.Q4}
M(y)=(\mathbb L^{-1}\varphi)(y)=\frac{|y|}{2\pi\sqrt{-1}}\int_{C(y)}(yz^{-1};q)_\infty\,\varphi(z)\frac{dz}{z^2}.
\end{equation}

Recall that $e_q(x):=\dfrac1{(x;q)_\infty}$ and $E_q(x):=(-x;q)_\infty$ are two different $q$-analogues of the exponential function, and observe that these two functions serve as the kernels in  \eqref{eq3.Q3} and \eqref{eq3.Q4}, respectively. For this reason we may consider $\mathbb L$ and $\mathbb L^{-1}$ as a reasonable version of the \emph{$q$-Laplace transform on $\L$ and its inverse}, cf. \cite{Hahn} and  \cite[Section 3.1.1]{BC}. In turn, $\mathbb L_N$ and $\mathbb L_N^{-1}$ may be viewed as  a  \emph{truncated version} of $\mathbb L$ and $\mathbb L^{-1}$. 

Proposition \ref{prop3.A} shows that the $\mathbb L_N$-transform of the $q$-B-spline with $N$ knots admits a simple explicit expression. In the context of the classical B-splines, a similar result was discovered by Curry and Schoenberg, see Lemma 6 in \cite{CS}. That lemma plays a key role in the proof of Theorem 6 from their paper.

\medskip

Institute for Information Transmission Problems, Moscow, Russia

National Research University Higher School of Economics, Moscow, Russia

Email: olsh2007@gmail.com

\end{document}